\documentclass[11pt]{amsart}
\usepackage[margin=1in]{geometry}

\usepackage{amssymb}
\usepackage{amsthm}
\usepackage{amsmath}
\usepackage{mathrsfs}
\usepackage{amsbsy}
\usepackage[all]{xy}
\usepackage{bm}
\usepackage{hyperref}
\usepackage{tikz}
\usepackage{array}
\usepackage{float}
\usepackage{enumerate}
\usepackage{xcolor}
\usepackage{hhline}
\usepackage{mathtools}
\allowdisplaybreaks
\usepackage[noadjust]{cite}

\usepackage{caption}
\usepackage{subcaption}
\usepackage{diagbox}
\usepackage{tikz}
\usepackage{bbm}
\usepackage{booktabs}

\usepackage[noabbrev,capitalise]{cleveref}

\newenvironment{enumerate*}%
{\begin{enumerate}[(I)]%
\setlength{\itemsep}{10pt}%
\setlength{\parskip}{0pt}}%
{\end{enumerate}}

\newtheorem{theorem}{Theorem}[section]
\newtheorem{proposition}[theorem]{Proposition}

\newtheorem{conjecture}[theorem]{Conjecture}

\newtheorem{lemma}[theorem]{Lemma}

\theoremstyle{definition}

\newcommand{\bp}{\mathbf{p}}

\newcommand{\bn}{\mathbf{n}}
\newcommand{\bt}{\mathbf{t}}

\DeclareMathOperator{\Span}{span}
\DeclareMathOperator{\IS}{IS}

\title{New bounds for (weak) sequenceability in $\mathbb{Z}_k$}

\author[]{Simone Costa}
\address[Simone Costa]{DICATAM, Universit\`a degli Studi di Brescia, Via Branze~43, I~25123 Brescia, Italy}
\email{simone.costa@unibs.it}
\author[]{Stefano Della Fiore}
\address[Stefano Della Fiore]{DII, Universit\`a degli Studi di Brescia, Via Branze 43, 25123 Brescia, Italy}
\email{stefano.dellafiore@unibs.it}

\begin{document}

\begin{abstract}
A famous conjecture of Graham asserts that every set $A \subseteq \mathbb{Z}_p \setminus \{0\}$
can be ordered so that all partial sums are distinct.
Although this conjecture was recently proved for sufficiently large primes by Pham and Sauermann in~\cite{PM}, it remains open for general abelian groups, even in the cyclic case $\mathbb{Z}_k$.

For cyclic groups, the best known result is due to Bedert and Kravitz in~\cite{BK}, who proved - using a rectification and a two-step probabilistic approach - that the conjecture holds for any subset $A \subseteq \mathbb{Z}_k \setminus \{0\}$ such that
\[
|A| \le \exp\!\big(c(\log p)^{1/4}\big),
\]
for some constant $c>0$, where $p$ denotes the least prime divisor of $k$.

In this paper, we improve their bound using a rectification argument again, followed by a one-shot probabilistic approach, showing that the conjecture holds whenever
\[
|A| \le \exp\!\big(c(\log p)^{1/3}\big),
\]
thus improving the exponent $1/4$ from~\cite{BK}.

Moreover, the same one-shot approach adapts to the $t$-weak setting:
by imposing all local constraints at once and applying the Lov\'asz Local Lemma, we obtain the existence of a $t$-weak sequencing whenever
\[
t \le \exp\!\big(c(\log p)^{1/4}\big).
\]
\end{abstract}
\maketitle

\section{Introduction}
Let $A$ be a finite subset of an abelian group $(G,+)$. We say that an ordering
$a_1, \ldots, a_{|A|}$ of $A$ is \emph{valid} if its partial sums
$p_1=a_1, p_2=a_1+a_2, \ldots, p_{|A|}=a_1 + \cdots + a_{|A|}$ are pairwise distinct.
Moreover, this ordering is a \emph{sequencing} if it is valid and $p_i \neq 0$
for every $1 \le i \le |A|-1$. In this case, we say that $A$ is sequenceable.
If we relax the definition, we say that an ordering is a \emph{$t$-weak sequencing}
if for every $i \ne j$ with $1 \le |i-j| \le t$, the partial sums $p_i$ and $p_j$
are non-zero and distinct. In this case we say that $A$ is $t$-weak sequenceable.

In the literature, there are several conjectures about valid orderings and sequenceability.
We refer to \cite{CMPP18, OllisSurvey, PD} for an overview of the topic,
\cite{AKP, AL20, ADMS16, CDOR} for lists of related conjectures,
and \cite{BFMPY} for a treatment using rainbow paths.

Here, we explicitly recall Graham's conjecture, which states that every set of nonzero elements of $\mathbb{Z}_p$
has a valid ordering.

\begin{conjecture}[\cite{GR} and \cite{EG}]\label{conj:main}
Let $p$ be a prime. Then every subset $A \subseteq \mathbb{Z}_p \setminus\{0\}$ has a valid ordering.
\end{conjecture}

Until recently, the main results on this conjecture were for small values of $|A|$;
in particular, in~\cite{CDOR}, the conjecture was proved for sets $A$ of size at most $12$.
The first result involving arbitrarily large sets $A$ was presented by Kravitz~\cite{NK},
who used a rectification argument to show that Graham's conjecture holds for all sets $A$ of size
$|A|\le \log p/\log\log p$. A similar argument was also proposed (but not published) by Will Sawin~\cite{Sawin}.

Then, in \cite{BK}, Bedert and Kravitz improved - using a rectification and a two-step probabilistic approach - this upper bound to the following:
\begin{theorem}[\cite{BK}]\label{thm:mainBK}
Let $p$ be a large enough prime and let $c>0$. Then every subset $A \subseteq \mathbb{Z}_p \setminus\{0\}$
is sequenceable provided that
\[
|A| \leq \exp\!\big(c(\log p)^{1/4}\big).
\]
\end{theorem}
They explicitly state their result for $\mathbb{Z}_p$, but their approach can be easily adapted to a generic cyclic group.

Finally, in~\cite{PM}, Graham's conjecture was proved for all sufficiently large primes $p$. This result is a consequence of anticoncentration inequalities developed using a discrete Fourier approach that seems hard to adapt to the cyclic case. 

In this paper, we record two complementary developments. In one direction, we show how it is possible to improve the bound \cite{BK} again using a rectification argument and a one-shot probabilistic approach, and obtain
\begin{theorem}[Improved classical bound]\label{thm:mainClassic}
There exists a constant $c>0$ such that, denoted by $p$ the least prime divisor of $k$, then
every subset $A\subseteq \mathbb{Z}_k\setminus\{0\}$ is sequenceable provided that
\[
|A|\ \le\ \exp\!\big(c(\log p)^{1/3}\big).
\]
\end{theorem}
In particular, the one-shot scheme removes the need to separate the treatment of Type I and Type II intervals, and this structural simplification is what ultimately permits the sharper quantitative bound.

Here, we also show how this approach can be localized for the $t$-weak sequenceability problem. 

For this problem, in \cite{CDOR}, the authors proved that if the order of a group is $pe$ then all sufficiently large subsets of the non-identity elements are $t$-weakly sequenceable
when $p>3$ is prime, $e \le 3$ and $t\le 6$.
Then in \cite{CDFAP}, using a hybrid approach that combines Ramsey theory and the probabilistic method,
the authors proved that if the size of a subset $A$ of an abelian group $G$ is at least $t^{\alpha t}$
for some $\alpha > 2$ and $A$ does not contain $0$, then $A$ is $t$-weak sequenceable.

Here, using a one-shot procedure that involves the Lov\'asz Local Lemma, we obtain:

\begin{theorem}\label{thm:main2}
There exists a constant $c > 0$ such that, if $p$ is the least prime divisor of $k$, then every subset $A \subseteq \mathbb{Z}_{k} \setminus\{0\}$
is $t$-weak sequenceable provided that
\[
t \ \leq\ \exp\!\big(c(\log p)^{1/4}\big).
\]
\end{theorem}

\textbf{Notation.}
For a sequence $\mathbf{b}=b_1, \ldots, b_r$, let
\[
\IS(\mathbf{b}) := \{b_1+\cdots+b_j: 0 \leq j \leq r\}
\]
denote the set of initial segment sums of $\mathbf{b}$, and let $\overline{\mathbf{b}}:=b_r, \ldots, b_1$
denote the reverse of $\mathbf{b}$. In addition, we denote
\[
\IS_t(\mathbf{b}) := \{b_1+\cdots+b_j: 0 \leq j \leq t\}.
\]
Using standard asymptotic notation, we say that $f = \Theta(g)$ if there exist two absolute constants
$C_1, C_2 > 0$ such that $C_1 g \leq f \leq  C_2 g$. We also define $f(p) = o(g(p))$ if
$\lim_{p \to \infty} f(p) / g(p) = 0$.

\section{Proof of Theorem \ref{thm:main2} and the one-shot framework}

Let $G$ be an abelian group. A subset $D=\{d_1, \ldots, d_r\} \subseteq G$ is \emph{dissociated} if
\[
\epsilon_1 d_1+\cdots +\epsilon_r d_r \neq 0
\quad \text{for all $(\epsilon_1, \ldots, \epsilon_r) \in \{-1,0,1\}^r \setminus \{(0,\ldots,0)\}$}.
\]
Equivalently, $D$ is dissociated if all of the $2^{|D|}$ subset sums of $D$ are distinct.
The {\it dimension} of a subset $B\subseteq G$, written $\dim(B)$, is the size of the largest dissociated set contained in $B$.
The $\Span(B)$ of a subset $B\subseteq G$, is defined as
$$\Span(B):=\left\{\sum_{b\in B}\epsilon_b b:\epsilon_b\in\{-1,0,1\}\right\}$$
\subsection{Structure Theorem}
We begin by stating a variation of the Structure Theorem of Bedert and Kravitz, enunciated here for the rings $\mathbb{Z}_k$ in a form that also enforces the dissociated sets $D_j$ to have comparable size.

Throughout the $t$-weak part of this paper we set
\[
R := R(k)= c_1 (\log p)^{1/2},
\]
where $c_1>0$ is a sufficiently small absolute constant and $p$ is the least prime divisors of $k$.
The Structure Theorem is based on the following rectification Lemma, proved here for cyclic groups.
\begin{lemma}
If $B\subseteq \mathbb{Z}_k$ is a nonempty subset of dimension $dim(B)<R$, then there is some $\lambda\in \mathbb{Z}_k^{\times}$ such that the dilate $\lambda \cdot B$ is contained in the interval $(-\frac{k}{100|B|},\frac{k}{100|B|})$.
\end{lemma}
\proof
Let $D$ be a maximal dissociated set of $B$ and let $\Lambda=\{0,1,\dots,p-1\}\subset \mathbb{Z}_k$. It is clear that, given $\lambda_1,\lambda_2\in \Lambda$, $\lambda_1-\lambda_2\in \mathbb{Z}_k^{\times}$.

Due to the pigeonhole principle, there exists distinct $\lambda_1, \lambda_2 \in \Lambda$ such that $\|\lambda_1 x_i -\lambda_2 x_i\| \leq \frac{k}{p^{1/r}}$ for all $i\in [1,r]$.  Set $\lambda\vcentcolon=\lambda_1-\lambda_2$, we have that $\lambda\in \mathbb{Z}_k^{\times}$ and $\lambda(D)\subseteq  [-\frac{k}{p^{1/\dim(B)}},\frac{k}{p^{1/\dim(B)}}]$.
Since $B\subseteq \Span(D)$, we have
$$B \subseteq \left[-\dim(B)\frac{k}{p^{1/\dim(B)}},\dim(B)\frac{k}{p^{1/\dim(B)}}\right].$$
The thesis follows if we prove that 
\begin{equation}\label{goal}
\dim(B)\frac{k}{p^{1/\dim(B)}}<\frac{k}{100|B|}
\end{equation} 
as long as $c_1$ is chosen to be sufficiently small.

Set $h=dim(B)$, Equation \eqref{goal} is equivalent to
$$\log(h \frac{k}{p^{1/h}})=\log(h)+\log{k}-1/h\log{p}<\log{k}-log{100}-\log{|B|}.$$
Since $|B|<3^{\dim(B)}$, this relation is implied by
$$h\log{h}+2h\log{10}+h^2\log{3}<\log{p}.$$
Which holds since we have assumed that $h < R = c_1(\log{p})^{1/2}$.
\endproof
\begin{theorem}[Structure Theorem \cite{BK}]\label{prop:structure}
For every nonempty subset $A \subseteq \mathbb{Z}_k \setminus \{0\}$,
there is some $\lambda \in \mathbb{Z}_k^\times$ such that $\lambda \cdot A$ can be partitioned as
\[
\lambda\cdot A = P\cup N\cup\Big(\bigcup_{j=1}^{s} D_j\Big),
\]
where:
\begin{enumerate}[(i)]
\item the ``positive'' set $P$ is contained in $\big(0,\frac{k}{4|P\cup N|}\big)$,
the ``negative'' set $N$ is contained in $\big(-\frac{k}{4|P\cup N|},0\big)$,
and the element $\delta:=\sum_{j=1}^{s}\sum_{d\in D_j} d$ is contained in $(-\frac{k}{4},\frac{k}{4})$;
\end{enumerate}
and if $s>0$ then:
\begin{enumerate}[(i)]
\item[(ii)] $P \cup N$ is nonempty, and each $D_j$ is dissociated of size $|D_j| = \Theta(R)$;
\item[(iii)] $\delta \notin \{0\}\cup -P\cup -N$, and moreover $\delta \neq -\sum_{p \in P}p$ if $N$ is nonempty
and $\delta \neq -\sum_{n \in N}n$ if $P$ is nonempty;
\item[(iv)] $D_1 \cup D_s \cup \{\delta\}$ is dissociated.
\end{enumerate}
\end{theorem}

\subsection{Ordering $P$ and $N$}

Here we recall the following important proposition from \cite{BK}.

\begin{proposition}\label{prop:ordering-P-N}
Let $P \subseteq (0,\frac{k}{4|P\cup N|})$ and $N \subseteq (-\frac{k}{4|P\cup N|},0)$ be subsets of $\mathbb{Z}_k$, and let $\delta>0$ be contained in $(0,\frac{k}{4|P\cup N|})$; moreover, assume that $ \delta \neq -\sum_{n \in N}n$ if $P \neq \emptyset$.  Let $Y^+_1, \ldots, Y^+_m,Y^-_1, \ldots, Y^-_m \subseteq \mathbb{Z}$ be finite sets.  Then there are orderings $\bp$ of $P$ and $\bn$ of $N$ such that $\overline \bp, \delta, \bn$ is a sequencing and we have
\begin{align*}
    |\IS(\bp) \cap Y^+_j| \leq \inf_{L \in \mathbb{N}} \left(\frac{|Y^+_j|}{L}+L+4+4\sum_{i=1}^{j-1}|Y^+_i|\right)
\end{align*}
and 
\begin{align*}
    |\IS(\bn) \cap Y^-_j| \leq \inf_{L \in \mathbb{N}} \left(\frac{|Y^-_j|}{L}+L+4+4\sum_{i=1}^{j-1}|Y^-_i|\right)
\end{align*}
for all $1 \leq j \leq m$.
\end{proposition}

\subsection{Splitting, rearrangement, and a one-shot control of Type I and Type II}

We begin with the standard anti-concentration input for dissociated sets.

\begin{lemma}[Lemma 5.1 of \cite{BK}]\label{splittinglemma}
Let $D\subset G$ be a dissociated set, and let $D=D^{(1)} \sqcup D^{(2)} \sqcup D^{(3)} \sqcup D^{(4)}$
be a uniformly random partition of $D$ into four sets of equal size.
Then for every proper subset $I \subset [4]$ and every $x\in G$,
\[
\mathbb{P}\Big(\sum_{i \in I}\sum_{d \in D^{(i)}}d=x\Big)
\leq \binom{|D|}{|D| \cdot |I|/4}^{-1}
\leq \binom{|D|}{|D|/4}^{-1}.
\]
\end{lemma}

\medskip
Starting from a set $A$, we consider $D_1,\dots,D_s$ to be the dissociated sets appearing in the structural decomposition of $\lambda\cdot A$ provided by Theorem~\ref{prop:structure} and $P\subseteq (0,\frac{k}{4|P\cup N|})$, $N\subseteq (-\frac{k}{4|P\cup N|},0)$ the sets of positive and negative elements there defined. We split and rearrange the dissociated sets as follows.

\begin{enumerate}[(S1)]
\item For each $j\in\{2,\dots,s-1\}$, we partition $D_j$ into four equal parts
$D_j = D_j^{(1)}\sqcup D_j^{(2)}\sqcup D_j^{(3)}\sqcup D_j^{(4)}$ uniformly at random
as in Lemma~\ref{splittinglemma}. We do all these splittings independently.
\item For the endpoint blocks $D_1$ and $D_s$, we choose a \emph{uniform random permutation}
$\sigma_1$ of $D_1$ and define $D_1^{(1)},\dots,D_1^{(4)}$ as the four consecutive segments (equal size)
of the list $\sigma_1$. Likewise, we choose a uniform random permutation $\sigma_s$ of $D_s$
and define $D_s^{(1)},\dots,D_s^{(4)}$ as consecutive segments.
\end{enumerate}

Next, we place these newly formed sets, together with $P$ and $N$, in the deterministic order
\begin{align}\label{eq:splitandrearrange}
P, D_1^{(1)},D_1^{(2)},D_2^{(1)},D_2^{(2)},\ldots,D_{s}^{(1)},D_s^{(2)},
D_1^{(3)},D_1^{(4)},D_2^{(3)},D_2^{(4)},\ldots,D_s^{(3)},D_{s}^{(4)}, N.
\end{align}
Write $T_1,\dots,T_u$ (with $u=4s$) for the resulting sequence of dissociated sets in~\eqref{eq:splitandrearrange},
and set $\tau_j:=\sum_{t\in T_j}t$. We also denote
\begin{align*}
\sum_{\leq M} (T_j) := \Big\{ \sum_{t \in S} t : S \subseteq T_j,\ |S| \leq M \Big\}, \\
\sum_{= M} (T_j) := \Big\{ \sum_{t \in S} t : S \subseteq T_j,\ |S| = M \Big\}\,.
\end{align*}

Fix an integer $K = c_2 R^{1/2}$ where $c_2$ is a positive small enough constant.
A proper nonempty interval $I \subset [1, |A|]$ is of \emph{Type II} if it contains between $K$ and $|T_j|-K$ elements of some block $T_j$. Otherwise $I$ is of \emph{Type I}.

Now, set
\begin{align*}\label{Y_jdefinition}
    Y^+_j\vcentcolon=-\sum_{= j}(D_1)\cup\left(-\delta+\sum_{= j}(D_s)\right) \quad \text{and} \quad Y^-_j\vcentcolon=-\sum_{= j}(D_s)\cup\left(-\delta+\sum_{= j}(D_1)\right)
\end{align*} for each $1\leqslant j \leqslant K$, and apply Proposition \ref{prop:ordering-P-N}. This provides orderings $\bp$ of $P$ and $\bn$ of $N$ such that the sequence $\overline{\bp},\delta,\bn$ is a sequencing and such that the bounds in Proposition \ref{prop:ordering-P-N} hold.

We say that an ordering $t_1,\dots,t_{|T_1|}$ of $T_1$ is \emph{acceptable} if $$t_1+\dots+t_k\notin -\IS(\bp)\cup \left(\delta+\IS(\bn)\right) \quad \text{for all $1 \leq k \leqslant K$},$$ and say that an ordering $t_1,\dots,t_{|T_u|}$ of $T_u$ is \emph{acceptable} if $$t_1+\dots+t_k\notin -\IS(\bn)\cup \left(\delta+\IS(\bp)\right) \quad \text{for all $1 \leq k \leqslant K$}.$$

 We then state the following lemma, which is an improvement of Lemma~6.1 of \cite{BK}.

\begin{lemma}\label{lemm:permissible}
    Let $\bt_1$ be the order induced by $\sigma_1$ on $T_1$ and $\bt_u$ be the order induced by $\sigma_s$ on $T_u$. Then, we have that $\bt_1$ is acceptable with probability at least $0.99$ and $\bt_u$ $0.99$ is acceptable with probability at least $0.99$.
\end{lemma}
\begin{proof}
We prove only the statement for $\bt_1$ since the argument for $\bt_u$ is identical.  Let $\bt_1 = t_1, \ldots, t_{|T_1|}$ be our random ordering induced by $\sigma_1$ on $T_1$.  By the union bound, it suffices to show that $$\mathbb{P}(t_1+\cdots+t_k \in -\IS(\bp) \cup (\delta+\IS(\bn))) \leqslant 0.01K^{-1}$$
for each $1 \leqslant k \leqslant K$.
Fix some $1 \leqslant k \leqslant K$.  Then the quantity $t_1+\cdots+t_k$ is uniformly distributed on the set $\sum_{=k}(D_1)$, which has size $\binom{|D_1|}{k}$.
Then by Proposition \ref{prop:ordering-P-N}, with $L:=\lfloor |Y_k^+|^{1/2}\rfloor$, we have that
$$\left|\sum_{=k}(D_1) \cap (-\IS(\bp) \cup (\delta+\IS(\bn))) \right| = O\left( |Y_k^+|^{1/2}+\sum_{j<k} |Y_j^+| + 1 \right).$$
For $1 \leqslant k \leqslant K$ (recall that $K=c_2 R^{1/2}$) it gives
$$\left|\sum_{=k}(D_1) \cap (-\IS(\bp) \cup (\delta+\IS(\bn))) \right| = O\left( \binom{|D_1|}{k} \cdot \frac{K}{|D_1|} \right) = O\left( \binom{|D_1|}{k} \cdot c_2^2 K^{-1} \right).$$

It follows that for $1 \leqslant k \leqslant K$ we have that
$$\mathbb{P}(t_1+\cdots+t_k \in -\IS(\bp) \cup (\delta+\IS(\bn))) = O\left(c_2^2 K^{-1} \right)$$
is at most $0.01 K^{-1}$ as long as $c_2$ is sufficiently small.
\end{proof}

We say that a pair of partial orderings $t_1, \ldots, t_k$ of $T_{2j}$ and $t'_1, \ldots, t'_\ell$ of $T_{2j+1}$ is \emph{permissible} if $$t_1+\cdots+t_i + t'_1+\cdots+t'_j\neq 0 \quad \text{for all $(i,j)$}.$$

We recall the symmetric Lov\'asz Local Lemma, which we use in the $t$-weak regime.

\begin{lemma}[Lov\'asz Local Lemma (symmetric case)]\label{lem:LLL}
Let $E_{1},E_{2},\ldots, E_{m}$ be events in a probability space,
where each event $E_{i}$ is mutually independent of all the other events $E_{j}$
except for at most $D$, and $\mathbb{P}(E_{i})\leq P$ for all $1\le i\leq m$.
If $ePD\leq 1$, then $\Pr(\cap_{i=1}^{m}\overline{E_{i}})>0$.
\end{lemma}

\subsubsection*{A one-shot lemma}
We now state a lemma that simultaneously controls Type~I and Type~II intervals by sampling the splitting and the internal orderings in a single step.

\begin{lemma}[One-shot Type I/II control]\label{lem:oneshot}
Assume $t\le \exp(cK)$ for a sufficiently small absolute $c>0$. Let
$D_1,\dots,D_s\subseteq \mathbb{Z}_k$ be dissociated sets, each of size $\Theta(R)$,
such that $D_1\cup D_s\cup\{\delta\}$ is dissociated, where $\delta:=\sum_{j=1}^s\sum_{d\in D_j} d$.
Let $\bp$ and $\bn$ be sequences over $\mathbb{Z}_k$ and assume that $\overline{\bp},\delta,\bn$
is a $t$-weak sequencing.

Fix an integer $K = c_2 R^{1/2}$ where $c_2$ is a small positive constant.
Choose the sets $T_1,\dots,T_u$ by the splitting procedure above, and then choose
orderings $\bt_1,\dots,\bt_u$ of the blocks $T_1,\dots,T_u$ as follows:
\begin{enumerate}[(a)]
\item $\bt_1$ is the order induced by $\sigma_1$ on $T_1$ and $\bt_u$ is the order induced by $\sigma_s$ on $T_u$,
and we \emph{condition} on the event that both $\bt_1$ and $\bt_u$ are \emph{acceptable} with parameter $K$
(in the sense of Lemma \ref{lemm:permissible}).
\item For each internal adjacent pair $(T_{2j},T_{2j+1})$ we sample uniformly at random the pair $(\bt_{2j},\bt_{2j+1})$ from the permissible pairs of orderings of length $K$ (as done in \cite[Lemma 6.2]{BK}).
\end{enumerate}
Let $a_1,\dots,a_{|A|}$ be the concatenation
\[
a_1,\dots,a_{|A|} \ :=\ \overline{\bp},\ \bt_1,\dots,\bt_u,\ \bn.
\]

Then:
\begin{enumerate}[(1)]
\item \textup{(Type II anti-concentration)} For every Type~II interval $I\subset[1,|A|]$ with $|I|\le t$,
\[
\mathbb{P}\Big(\sum_{i\in I} a_i = 0\Big)\ \le\ \exp\!\big(-\Theta(K\log R)\big).
\]
\item \textup{(Type I anti-concentration)} For every Type~I interval $I\subset[1,|A|]$ with $|I|\le t$,
\[
\mathbb{P}\Big(\sum_{i\in I} a_i = 0\Big)\ \le\ \exp\!\big(-\Theta(R)\big).
\]

\item \textup{($t$-weak existence via LLL)}
With positive probability (hence there exists a choice of the random splittings and orderings)
every interval $I\subset[1,|A|]$ with $|I|\le t$ has nonzero sum, and therefore
$a_1,\dots,a_{|A|}$ is a $t$-weak sequencing.
\end{enumerate}
\end{lemma}

\begin{proof} 
The Type~II anti-concentration inequality follows exactly by the same argument used in~\cite[Lemma~6.3]{BK}.

\emph{Type I.} Let us consider the following set of events:

\begin{enumerate}[(i)]
    \item For each proper nonempty interval $[i,j] \subseteq [u]$, $|i-j| \leq t$,
    $$ 0 \in \left(\sum_{\leq K}(T_{i-1}) \cup -\sum_{\leq K}(T_i) \right)+\tau_i+\cdots+\tau_j+ \left(-\sum_{\leq K}(T_j) \cup \sum_{ \leq K}(T_{j+1}) \right)$$
    (with the convention that $T_0=T_{u+1}=\emptyset$);
    \item For each $1 \leq j \leq \min\{t, u-1\}$,
    $$0 \in \IS_{t}(\bp)+\tau_1+\cdots+\tau_j+ \left(-\sum_{\leq K}(T_j) \cup \sum_{ \leq K}(T_{j+1}) \right);$$
     and for each $\max\{2, u-t+1\} \leq j \leq u$,
    $$0 \in \IS_{t}(\bn)+\tau_u+\cdots+\tau_j+ \left(-\sum_{\leq K}(T_j) \cup \sum_{ \leq K}(T_{j-1}) \right).$$
\end{enumerate}
By the deterministic order \eqref{eq:splitandrearrange}, there exists $\ell$ such that each proper nonempty interval $J \subseteq [u]$ contains some but not all of $D_\ell^{(1)},\dots,D_\ell^{(4)}$; hence $\sum_{j\in J}\tau_j$ includes a subset-sum of $D_\ell$ of fixed size,
which is distributed uniformly among all such subset-sums. By dissociativity, all these sums are distinct, so each value is attained with probability at most $\binom{|D_\ell|}{|D_\ell|/4}^{-1}\le \exp(-\Theta(R))$ thanks to Lemma~\ref{splittinglemma}. Then since $|T_j| = \Theta(R)$ and $K=o(R)$ we have that $|\sum_{\leq K} (T_{j})| = \exp(o(R))$. Also note that $|\IS_t(\bp)|, |\IS_n(\bn)| \leq t = \exp(o(R))$, therefore we find that conditions $(i)$, $(ii)$ have probability of at most $\exp(-\Theta(R))$. This handles all Type I intervals except the ones starting in the last $K$ elements of $\mathbf{t}_{2j}$ and ending in the first $K$ elements of $\mathbf{t}_{2j+1}$ and the ones starting in $\overline{\mathbf{p}}$ and ending in the first $K$ elements of $\mathbf{t}_1$ or the ones starting last $K$ elements of $\mathbf{t}_u$ and ending in $\mathbf{n}$ but these lat two cases are avoided since we are conditioning on the event that both $\bt_1$ and $\bt_u$ are acceptable and the pairs $(\bt_{2j}, \bt_{2j+1})$ are permissible.

\emph{$t$-weak via LLL.} Consider bad events $E_I$ for intervals $I\subset[1,|A|]$, $|I|\le t$, defined by
$\sum_{i\in I} a_i=0$. By the previous two parts, $\mathbb{P}(E_I)\le P$ with
$P\le \exp(-\Theta(K\log R))$. We now bound the dependency degree in the Lov\'asz Local Lemma.

For any interval $I\subset [1,|A|]$ and block $T_i\subseteq D_{\ell}$, the bad event $E_I$ depends on
\begin{itemize}
\item the random ordering of $T_i$;
\item the random \emph{splitting variables} used to generate the four pieces
$D_\ell^{(1)},D_\ell^{(2)},D_\ell^{(3)},D_\ell^{(4)}$;
\item the random ordering of $T_{i'}\subseteq D_{\ell'}$ where $(T_i,T_{i'})$ are adjacent and conditioned by the permissibility condition;
\item the random \emph{splitting variables} used to generate the four pieces
$D_{\ell'}^{(1)},D_{\ell'}^{(2)},D_{\ell'}^{(3)},D_{\ell'}^{(4)}$.
\end{itemize}

Fix $E_I$. Since $|I|\le t$, the interval $I$ intersects at most $t$ blocks.
For each dissociated set $D_\ell$ touched by $I$, there are at most $8$ associated pieces, each of size $\Theta(R)$.
The number of intervals of length at most $t$ that intersect a fixed block of size $\Theta(R)$ is $O(tR+t^2)$.
Therefore the total number of bad events that can fail to be mutually independent from $E_I$
is at most
\[
D = O\big(t(tR+t^2)\big)=O(t^2R+t^3).
\]
Thus, by the symmetric Lov\'asz Local Lemma, it suffices to verify
\[
ePD \le 1.
\]
For $t\le \exp(cK)$ and $c>0$ sufficiently small, this holds since
$P\le \exp(-\Theta(K\log R))$ and $D=\exp(O(\log t+\log R))$.
Hence $\Pr\left(\cap_I \overline{E_I}\right)>0$.
\end{proof}

\subsection{Proof of Theorem \ref{thm:main2}}
We now complete the proof of Theorem~\ref{thm:main2}.
Let $A\subseteq \mathbb{Z}_k\setminus\{0\}$. Apply Theorem~\ref{prop:structure} to obtain a dilation $\lambda$
and a decomposition $\lambda A = P\cup N\cup (\cup_{j=1}^s D_j)$ with the stated properties.
By~\cite[Proposition~4.1]{BK} (see Lemma \ref{prop:ordering-P-N} of this paper), we can choose orderings
$\bp$ of $P$ and $\bn$ of $N$ such that $\overline{\bp},\delta,\bn$ is a sequencing.

Now set
\[
K := c_2 R^{1/2} \ =\ \Theta\big((\log p)^{1/4}\big),
\]
with $c_2>0$ sufficiently small, and apply Lemma~\ref{lem:oneshot}.
Part~(3) of Lemma~\ref{lem:oneshot} gives existence of a $t$-weak sequencing provided
$t\le \exp(cK)$ for a sufficiently small absolute constant $c>0$, which yields
$t \le \exp(c(\log p)^{1/4})$ after renaming constants.  Therefore, if $p$ is large enough (i.e. $p\geq \bar{p}$),  with positive probability, the sampled ordering contains no
nontrivial zero-sum interval, i.e.\ it is a $t$-weak sequencing. Moreover we can chose $c$ small enough so that $\exp\!\big(c(\log p)^{1/4}\big)<2$ for any prime $p<\bar{p}$. 
This proves the theorem.
\qed

\section{Improved classical sequenceability (Proof of Theorem \ref{thm:mainClassic})}
We explain how the one-shot control of Type~I and Type~II intervals yields the improved classical bound.
Here we work in the classical (non-$t$-weak) setting: we must avoid \emph{all} nontrivial intervals.

\begin{proof}[Proof of Theorem~\ref{thm:mainClassic}]
Let $A\subseteq \mathbb{Z}_k\setminus\{0\}$, and let us denote by $p$ the least prime divisor of $k$.
Apply the rectification/structure step (as in~\cite{BK}) with a choice of parameters that yields dissociated blocks of size
$$
 R=R(A,k):=c_1\max\left((\log p)^{1/2},\frac{\log p}{\log |A|}\right)\,.
$$
We recall that the Structure Theorem of \cite{BK} holds for this value of $R$ (the previous definition of $R$ was functional to the weak-sequenceability result).
Here we run the one-shot construction of Lemma~\ref{lem:oneshot} with
\[
K := c_2 R^{1/2}
\qquad (c_2>0\ \text{sufficiently small}).
\]

In the classical setting, we use a union bound over all nontrivial intervals $I\subset[1,|A|]$.
There are at most $|A|^2$ such intervals. For Type~II intervals, Lemma~\ref{lem:oneshot}(1) gives
\[
\mathbb{P}\Big(\sum_{i\in I}a_i=0\Big)\ \le\ \exp\!\big(-\Theta(K\log R)\big)
\ =\ \exp\!\big(-\Theta(R^{1/2}\log R)\big).
\]
Similarly for Type~I intervals, Lemma~\ref{lem:oneshot}(2) gives
\[
\mathbb{P}\Big(\sum_{i\in I}a_i=0\Big)\ \le\ \exp\!\big(-\Theta(R)\big).
\]
Hence
\[
\mathbb{P}(\exists\ \text{nontrivial }I:\ \sum_{i\in I}a_i=0)
\ \le\ |A|^2\exp\!\big(-\Theta(R^{1/2}\log R)\big) \ +\ |A|^2\exp\!\big(-\Theta(R)\big).
\]
Since $|A| \leq \exp(c (\log p)^{1/3})$ for $c>0$ sufficiently small then $R$ is at least $\Theta\big((\log p)^{2/3}\big)$, we obtain that the right-hand side is $o(1)$. Therefore, if $p$ is large enough (i.e. $p\geq \bar{p}$),  with positive probability, the sampled ordering contains no
nontrivial zero-sum interval, i.e.\ it is a sequencing. Moreover we can chose $c$ small enough so that $\exp\!\big(c(\log p)^{1/3}\big)<2$ for any prime $p<\bar{p}$. 
This proves the theorem.
\end{proof}

\end{document}